\documentclass[11pt]{article}

\usepackage{amsfonts}
\usepackage{amsmath}
\usepackage{amssymb}
\usepackage{amscd}
\usepackage{amsthm}
\usepackage{indentfirst}
\usepackage[all,arc]{xy}
\usepackage{enumerate}
\usepackage{mathrsfs}
\usepackage[hmargin=3cm,vmargin=3cm]{geometry}

\newtheorem{thm}{Theorem}
\newtheorem{prop}{Proposition}[section]

\theoremstyle{definition}

\theoremstyle{remark}

\newcommand{\R}{{\mathbb{R}}}

\newcommand{\HH}{{\mathbb{H}}}

\newcommand{\del}{\partial}

\newcommand{\sm}[1]{C^\infty(#1)}
\newcommand{\av}[1]{\langle#1\rangle}
\newcommand{\delbar}{\overline{\partial}}
\newcommand{\Sum}{\Sigma}

\newcommand{\normL}[3]{||#1||_{L^{#2}(#3)}}
\newcommand{\normC}[3]{||#1||_{C^{#2}(#3)}}

\newcommand{\om}{\omega}
\newcommand{\Om}{\Omega}

\newcommand{\ep}{\epsilon}

\usepackage{amsmath,amsthm}
\usepackage{amssymb,latexsym}
\usepackage{amscd}

\def\eps{\varepsilon}

\def\Ome{\Omega}

\def\Lam{\Lambda}

\def\to{\rightarrow}
\def\pt{\partial}

\def\RR{\mathbb{R}}

\def\HH{\mathbb{H}}

\swapnumbers

\theoremstyle{theorem}
\newtheorem{theorem}{Theorem}

\newtheorem{lemma}[theorem]{Lemma}

\newtheorem{wrappingup}[theorem]{Wrapping up}
\theoremstyle{definition}

\newtheorem{definition}[theorem]{Definition}
\newtheorem{remark}[theorem]{Remark}

\newtheorem{construction}[theorem]{Construction}

\theoremstyle{proposition-definition}
\newtheorem{proposition-definition}[theorem]{Proposition-Definition}

\numberwithin{equation}{section}

\renewcommand{\thefootnote}{\alph{footnote}}

\begin{document}

\title{\textbf{On a uniform estimate for the \\ quaternionic Calabi problem}}

\author{\textsc{Semyon Alesker$^{a}$ and Egor Shelukhin$^{b}$} \\
School of Mathematical Sciences \\
Tel Aviv University\\
69978 Tel Aviv, Israel \\
semyon@post.tau.ac.il, egorshel@post.tau.ac.il}

\footnotetext[1]{Partially supported by the Israel Science Foundation grant 701/08.} \footnotetext[2]{Partially supported by the Israel Science Foundation grant 509/07.}

\renewcommand{\thefootnote}{\arabic{footnote}}

\date{}

\maketitle

\begin{abstract}
We establish a $C^0$ \textit{a priori} bound on the solutions of
the quaternionic Calabi-Yau  equation (of Monge-Amp$\grave{\text{e}}$re type) on compact HKT manifolds with a
locally flat hypercomplex structure. As an intermediate step, we prove a quaternionic version of the Gauduchon theorem.
\end{abstract}

\tableofcontents

\section{Introduction and main results}
\subsection{HKT manifolds}

We first recall some definitions and preliminaries from HKT geometry.

\begin{definition}
A smooth manifold $M^{4n}$ with a triple of complex structures $I,J,K$ such that $IJ=-JI=K$ will be called a \textit{hypercomplex manifold}.
\end{definition}

\begin{remark}
Hypercomplex manifolds were explicitly introduced by Boyer \cite{boyer}.
\end{remark}

\begin{definition}
A hypercomplex manifold $(M^{4n},I,J,K)$ with a Riemannian
metric $\rho$ that is invariant under the three complex structures $(I,J,K)$ will be called a \textit{hyper-Hermitian} manifold.
\end{definition}

\begin{definition}({\it The $\del_J$-operator})
Consider the complex manifold $(M,I)$. Denote by $\Om^{p,q}_I(M)$
the space of $(p,q)$-forms on $(M,I)$. Let $\del:\Om^{p,q}_I(M) \to
\Om^{p+1,q}_I(M)$ and $\delbar:\Om^{p,q}_I(M) \to \Om^{p,q+1}_I(M)$
be the usual $\del$ and $\delbar$ operators on the complex manifold
$(M,I)$. Define $$\del_J := J^{-1}\circ \delbar \circ J.$$ Then
\cite{hkt-mflds-supersymmetry} (a)$\; J:\Om^{p,q}_I(M) \to
\Om^{q,p}_I(M)$,(b)$\;  \del_J:\Om^{p,q}_I(M) \to \Om^{p+1,q}_I(M)$
and (c)$\; \del\del_J=-\del_J\del$.
\end{definition}

Given a hyper-Hermitian manifold $(M,I,J,K,\rho)$ consider the differential form $$\Om:= - \om_J + \sqrt{-1}\om_K$$ where $\om_L(A,B):=\rho(A,B \cdot L)$ for any $L \in \HH$ with $L^2= -1$ and any two vector fields $A,B$ on $M$. We use the following definition of an HKT-metric.

\begin{definition}\label{definition HKT metric} The metric $\rho$ on $M$ is called an \textit{HKT-metric} if $$\del \Om = 0.$$
\end{definition}

\begin{remark} HKT manifolds were introduced in the physical literature by
Howe and Papadopoulos \cite{howe-papa}. For the mathematical treatment see Grantcharov-Poon \cite{_Gra_Poon_} and Verbitsky \cite{hkt-mflds-supersymmetry}. The original definition of HKT-metrics in \cite{howe-papa}
was different but equivalent to Definition \ref{definition HKT metric} (the latter was given in \cite{_Gra_Poon_}).
\end{remark}

\begin{remark}
The classical hyper-Kahler metrics (i.e. Riemannian metrics with
holonomy contained in the group $Sp(n)$) form a subclass of
HKT-metrics. It is well known that a hyper-Hermitian metric $\rho$
is hyper-Kahler if and only if the form $\Om$ is closed, or
equivalently $\del \Om  = 0$ and $\delbar \Om = 0.$
\end{remark}

\begin{definition} A form $\om \in \Om^{2k,0}_I(M)$ that satisfies $J \om = \overline{\om}$ will be called a \textit{real (2k,0)-form} \cite{hkt-mflds-supersymmetry}. The space of real $(2k,0)$-forms will be denoted by $\Om^{2k,0}_{I,\R}(M).$
\end{definition}

\begin{definition}({\it The $t$ isomorphism} \cite{hkt-mflds-supersymmetry}) Given a right vector space $V$ over $\HH$, there is an $\RR$-linear
isomorphism $t:\Lambda^{2,0}_{I,\R}(V) \to S_\HH(V)$ from the space
of the real $(2,0)$-forms on $V$ to the space of symmetric forms on
$V$ that are invariant with respect to the action of the group
$SU(2)$ of unit quaternions, given by $t(\eta)(A,A):=\eta(A,A\circ
J)$ (\cite{hkt-mflds-supersymmetry}). We shall denote by the same
letter the induced isomorphism $t:\Om^{2,0}_{I,\R}(M) \to S_\HH(M)$,
where the target is the space of global sections of the bundle with
the fiber $S_\HH(M)_x = S_\HH(T_x M)$ over a general point $x \in
M$.
\end{definition}

\begin{definition} Forms $\eta \in \Om^{2,0}_{I,\R}(M)$ satisfying $t(\eta) > 0$ and $t(\eta) \geq 0$ will be called \textit{strictly positive} and \textit{positive} respectively. We also call a form $\Theta \in \Om^{2n,0}_{I,\R}(M)$ \textit{strictly positive} if it equals $f\cdot \Om^n$ for a strictly positive function $f \in \sm{M,\R}$. This definition does not depend on the choice of the hyper-Hermitian form $\Om$, because the real line bundle $\Lambda^{2n,0}_{I,\R}$ is canonically oriented, and strict positivity of $\Theta$ is equivalent to $\Theta_x$ being a strictly positive element in the the fiber over $x$ for all $x \in M$. We refer to \cite{_Alesker_Verbitsky_,QuaternionicCalabiYau1} and references therein for further discussion of the notion of positivity.
\end{definition}

\begin{remark}
On a hyper-Hermitian manifold the metric $\rho$ and the form $\Om =  - \om_J + \sqrt{-1}\om_K$ satisfy $t(\Om)=\rho$ and are therefore mutually defining.
\end{remark}

The following is a hypercomplex analogue of the $\del\delbar$-lemma on complex manifolds.

\begin{prop}({\it The $\del\del_J$-lemma} cf. \cite{_Alesker_Verbitsky_}) A form $\Om \in \Om^{2,0}_{I,\R}(M)$ satisfies $\del \Om = 0$ if and only if it can be locally represented as $\Om = \del\del_J f$ for a function $f \in C^\infty(M,\R)$.
\end{prop}

\begin{definition} A function $u \in \sm{M,\R}$ on a hypercomplex manifold $(M,I,J,K)$ for which $\del\del_J u$
is (strictly) positive will be called \textit{(strictly)
plurisubharmonic}.
\end{definition}

\begin{definition}
On any hypercomplex manifold there exists a unique connection
$\nabla$ with zero torsion that satisfies $\nabla I=\nabla J=\nabla
K =0$. It will be called \textit{the Obata connection} \cite{obata}
of the hypercomplex manifold.
\end{definition}

For a quadratic form $Q$ on a right vector space $V$ over the quaternions we shall denote by $\av{Q}_u$ its average over the action of the group $SU(2)$ of unit quaternions, that is $$\av{Q}_u(x):=\int_{SU(2)} Q(x \circ L) d\nu(L),$$ for the Haar probability measure $\nu$ on $SU(2)$. We shall denote by $Q_+$ the projection $Q_+:=Q - \av{Q}_u$ of $Q$.

\begin{lemma}\label{average over SU(2) of a quadratic form}
For a quadratic form $Q$ on a right vector space $V$ over the quaternions $$\av{Q}_u(x) = \frac{1}{4}(Q(x) + Q(x \circ I) + Q(x \circ J) + Q(x \circ K)).$$
\end{lemma}

The lemma follows by noting that both sides of the equality are $SU(2)$-invariant on one hand, and both their averages over $SU(2)$ equal $\av{Q}_u(x)$ on the other.

\begin{definition} For a quaternion $q\in \HH$ written in the
standard form
\[q= t + x \cdot i + y \cdot j + z \cdot k\] define the Dirac-Cauchy-Riemann operator $\frac{\del}{\del \overline{q}}$
on an $\HH$-valued function $F$ by
\[\frac{\del}{\del \overline{q}}F = \frac{\del}{\del t}F + \frac{\del}{\del x}F + \frac{\del}{\del y}F + \frac{\del}{\del z}F,\] and its quaternionic conjugate $\frac{\del}{\del q}$ by \[\frac{\del}{\del q}F = \overline{\frac{\del}{\del \overline{q}}\overline{F}} = \frac{\del}{\del t}F + \frac{\del}{\del x}F + \frac{\del}{\del y}F + \frac{\del}{\del z}F.\]
\end{definition}

Note that the quaternionic Hessian $\displaystyle Hess_\HH (u) := (\frac{\del^2
u}{\del \bar q_i \del q_j})$ is the average of the quadratic
form $D^2 u$ over $SU(2)$, as a short computation shows. As for
different $i$ and $j$ the operators $\frac{\del}{\del q_i}$ and
$\frac{\del}{\del \overline{q_j}}$ commute, this matrix is
hyper-Hermitian, namely satisfies the following definition.

\begin{definition} We call a quaternionic matrix $A = (a_{ij})_{i,j = 1}^n$ \textit{hyper-Hermitian} if $a_{ji} = \overline{a_{ij}}$ for the quaternionic conjugation $\HH \ni q \mapsto \overline{q} \in \HH.$
\end{definition}

We shall also use a version of a determinant defined for
hyper-Hermitian quaternionic matrices referring for further details,
properties, and references to \cite{alesker-bsm}. Considering $\HH$
as an $\R$-linear vector space of dimension $4$ we have the
embedding $0 \to Mat(n,\HH) \to Mat(4n,\R)$ of matrix $\R$-algebras.
Denote by $\mathcal{H}_n$ the image of the subspace of
hyper-Hermitian matrices.

\begin{prop}
There exists a polynomial $P$ on $\mathcal{H}_n$ such that $P^4 = det|_{\mathcal{H}_n}$ and $P(Id)=1$. Moreover $P$ is uniquely determined by these properties, is homogenous of degree $n$ and has integer coefficients.
\end{prop}

\begin{definition}
For a hyper-Hermitian matrix $A$ the \textit{Moore determinant} $det(A):= P (^\R A) \in \R$ for $^\R A$ the matrix in $\mathcal{H}_n$ corresponding to $A.$
\end{definition}

\begin{thm}
The Moore determinant restricts to the usual determinant on the $\R$-subspace of complex Hermitian matrices.
\end{thm}

On flat HKT manifolds, the Moore determinant of $Hess_\HH(u)$ can be naturally identified (up to a positive multiplicative
constant) with $(\del\del_J u)^n$, see \cite[Corollary
4.6]{_Alesker_Verbitsky_}.
%add Moore determinant = (\del \del_J u)^n

\subsection{The quaternionic Calabi-Yau equation}

A quaternionic version of the classical Calabi problem was
introduced by M. Verbitsky and the first named author in
\cite{QuaternionicCalabiYau1}. It says that if $(M,I,J,K,\Ome)$ is a
compact HKT-manifold and $f\in C^\infty(M,\RR)$ is a smooth real
valued function, then there exists a constant $A>0$ such that the
equation
$$(\Om + \del\del_J \phi)^n = A e^f \Om^n$$ has a solution $\phi \in
C^\infty(M,\R)$. Analogously to the K\"ahler case, if such a $\phi$
exists then $\Om + \del\del_J \phi$ is an HKT form. This equation is
non-linear and elliptic of second order. The uniqueness of a
solution is shown in \cite{QuaternionicCalabiYau1} and the existence of solutions was conjectured.
Notice that recently M. Verbitsky \cite{verbitsky-2009} has found a geometric
interpretation of this equation.

In \cite{QuaternionicCalabiYau1} it was also shown that a solution
of the above equation satisfies a $C^0$ a priori estimate provided
that the holonomy of the Obata connection of $(M,I,J,K)$ is
contained in the subgroup $SL_n(\HH)\subset GL_n(\HH)$. The main
result of this paper is to show the $C^0$ estimate under different
assumptions on the manifold, namely for the case of locally flat hypercomplex
structure (equivalently, whenever the Obata connection is flat).

In the recent preprint \cite{alesker-2011} the above conjecture was solved by the first named author under the additional
assumption that $M$ admits a hyperK\"ahler metric compatible with the underlying
hypercomplex structure.

\hfill

A different quaternionic version of the Monge-Amp\`ere equation on the flat quaternionic space $\mathbb{H}^n$ was introduced earlier by the
first named author \cite{alesker-jga-03} (based also on \cite{alesker-bsm}) where he proved the solvability of the Dirichlet
problem under appropriate assumptions.

\subsection{A-priori bounds}

In this paper we show uniform \textit{a priori} bounds on the solution $\phi$ of the quaternionic Calabi-Yau equation on a locally flat HKT manifold. That is - we show that there exists a constant $C$ depending only on $(M,I,J,K,\Om)$ and $f$ such that the solution $\phi$ normalized e.g. by $\max_M \phi = 0$ satisfies $$||\phi||_{L^\infty} \leq C.$$

Such bounds were shown in \cite{QuaternionicCalabiYau1} under the condition that there exists a holomorphic $(2n,0)$-form on $M$ (with respect to the complex structure $I$) that is nowhere vanishing. We show these bounds in the case of general locally flat HKT manifolds.

%This case should contain new examples as there are discrete subgroups $\Gamma$ of $GL(n,\HH)$ that are not in $SL(n,\HH)$ that provide us with the locally flat HKT manifolds $\HH^n/\Gamma$ that do not possess such a non-vanishing holomorphic $(2n,0)$.

More specifically we prove the following.

\begin{thm}\label{locally flat}
Let $(M,I,J,K,\Om)$ be a compact HKT manifold such that the Obata
connection is flat. Let $\phi \in \sm{M,\R}$ be a solution of the
equation
$$(\Om + \del\del_J\phi)^n = f \Om^n; \;f>0$$ with the normalization
condition $\max_M \phi = 0$. Then $$||\phi||_{L^\infty} \leq C,$$
for a constant $C$ depending only on $(M,I,J,K,\Om)$ and
$||f||_{L^\infty}$.
\end{thm}

\begin{remark}\label{R:rem18}
Let us describe an example of an HKT manifold satisfying the
assumptions of Theorem \ref{locally flat} which does not satisfy the
assumptions of the uniform estimate obtained in
\cite{QuaternionicCalabiYau1}. This example was kindly mentioned to
us by the anonymous referee; it appears explicitly in \cite[Section
4.3]{_Gra_Poon_}.

Let us consider the space $\HH^n\backslash\{0\}$ with the standard
flat hypercomplex structure. Consider the HKT metric corresponding
to the form $\Ome=\partial\partial_J (\ln r)$ where $r$ is the
Euclidean distance to the origin; the condition of positivity can be
easily checked by a direct computation.

Let us fix a real number $q> 0,\,q \neq 1$. The multiplication by $q$
acts on $\HH^n\backslash\{0\}$; this action preserves the
hypercomplex structure and the HKT metric. Hence it induces an HKT
structure on the quotient manifold $(\HH^n\backslash\{0\})/\langle q \rangle$
where $\langle q \rangle$ denotes the group generated by the multiplication by $q$.
This quotient is a compact HKT manifold; it is called the
quaternionic Hopf manifold. It is well known (and not hard to check)
that this manifold, equipped with the complex structure $I$, has no
non-zero holomorphic $(2n,0)$-forms. In particular the uniform
estimate of \cite{QuaternionicCalabiYau1} does not apply to it.
\end{remark}

\section*{Acknowledgements}
We thank V. Palamodov, I. Polterovich, L. Polterovich, A. Pulemotov,
and M. Verbitsky for useful discussions. We also thank the anonymous
referee for suggesting to us the example of the Hopf manifold
described in Remark \ref{R:rem18}.

\section{Proof}
Here we present the proof of Theorem \ref{locally flat} assuming several important propositions
that we prove later. The method of the proof is heavily based on the paper \cite{BlockiUniform} of Z. Blocki.
We use the generic notation $const$ to denote any positive constant that depends only on $n,M,I,J,K,\Ome$.

\begin{proof}(Theorem \ref{locally flat})
%We shall denote by $const$ any positive constant that depends only
%on $(M^{4n},I,J,K,\Om)$ and $n$.

{\it Step 1: the $L^1$ bound.}

As proven in \cite{QuaternionicCalabiYau1} $\Om + \del\del_J \phi > 0$ in this case, or in other words it is an HKT form for a certain HKT metric. Whence by Proposition \ref{P:L1} that we prove in Section \ref{L1 bound section} below, with the normalization $\max_M\phi=0$, the norm $||\phi||_{L^1}$ is bounded by a constant depending on $M,I,J,K,\Ome$.

{\it Step 2: The Taylor expansion of the potential of the metric.}

We first formulate and prove the exact statements that allow us to make our calculations uniform in an appropriate sense, and then make the calculations.

\begin{lemma}\label{good r, a and g}
Given an HKT-manifold $(M,I,J,K,\Om)$ there exist $r>0$ and $a>0$ such that $\forall z_0 \in M$ $\exists g \in \sm{B(z_0,2r)}$ such that
\begin{enumerate}
\item $g < 0$ on $B(z_0,2r)$,
\item $\Om = \del\del_J g$ on $B(z_0,2r)$,
\item $\inf_{B(z_0,2r) \setminus B(z_0,r)}g \geq \inf_{B(z_0,r)}g + a$,
\item $\inf_{B(z_0,r)}g = g(z_0),$
\item $\normC{g}{10}{B(z_0,2r)} \leq const.$
\end{enumerate}
\end{lemma}

\begin{proof}(Lemma \ref{good r, a and g})
\emph{Claim 1. Suitable atlas:} There exists a finite open cover $\{U_i\}_{i \in I}$ of $M$ such that
\begin{enumerate}
\item $U_i \Subset V_i$ for open sets $V_i$ isomorphic as hypercomplex manifolds to open subsets in $\mathbb{H}^{n}$,
\item there exist functions $g_i \in \sm{V_i,\R}$ such that $\Om = \del\del_J g_i$ on $V_i$,
\item $\normC{g_i}{20}{U_i} < const,$
\item there exists $r_0>0$ such that for any point $z_0\in M$ there
exists $i_0\in I$ satisfying $B(z_0,2r_0)\Subset U_{i_0}$.
\end{enumerate}
This claim is an easy consequence of Proposition 1.15 of
\cite{_Alesker_Verbitsky_}.

\emph{Claim 2. Uniform lower bound on the quaternionic Hessian:} There exists a positive constant $\epsilon$ such that $Hess_\HH(g_i) > \ep I$ on $U_i$ for all $i \in I$.

Indeed, since $g_i$ is strictly plurisubharmonic on $V_i$ we have
$Hess_\HH(g_i) > 0$ and as $U_i \Subset V_i$, there exist constants
$\ep_i$ such that $Hess_\HH(g_i) > \ep_i I$ on $U_i$. Take $\ep \leq
min_{i \in I} \ep_i$.

Given a point $z_0 \in M$, we shall now construct the function $g$ and
choose $r > 0$,$a > 0$. Consider $(U,\widetilde{g})=(U_{i_0},g_{i_0})$ with $B(z_0,2r_0)\Subset U_{i_0}$.

The Taylor expansion of $\widetilde{g}$ about $z_0$ (in flat coordinates on $V_i \Supset U_i$) gives $$\widetilde{g}(z_0 + h) = \widetilde{g}(z_0) + d_{z_0}\widetilde{g}(h) + D^2_{z_0}\widetilde{g}(h) + \Theta(h),$$ where $\Theta(h)=o(|h|^2).$ Now we split the quadratic form $D^2_{z_0}\widetilde{g}(h)$ into its invariant and complementary parts with respect to the action of $SU(2)$: $$D^2_{z_0}\widetilde{g}(h) = \av{D^2_{z_0}\widetilde{g}}_u(h) + (D^2_{z_0}\widetilde{g})_+(h) = Hess_\HH(\widetilde{g})_{z_0}(h) + (D^2_{z_0}\widetilde{g})_+(h).$$

Define $g$ by $$g(z_0 + h) := Hess_\HH(g)_{z_0}(h) + \Theta(h).$$

\emph{Claim 3. Equality of Hessians:} We have the pointwise equality $Hess_\HH(g) \equiv Hess_\HH(\widetilde{g})$ on $U$ (and therefore $\Om = \del\del_J g$ on $U$).

\begin{proof}(Claim 3)
It is sufficient to prove that the function $p(h):= \widetilde{g} - g = g(z_0) + d_{z_0}g(h) + (D^2_{z_0}g)_+(h)$ satisfies $Hess_\HH(p) \equiv 0$ on $U$. The quaternionic Hessian of the affine function $g(z_0) + d_{z_0}g(h)$ certainly vanishes. Therefore it is enough to prove that $Hess_\HH((D^2_{z_0}g)_+) \equiv 0$. And indeed $$Hess_\HH((D^2_{z_0}g)_+) = \av{D^2(D^2_{z_0}g - \av{D^2_{z_0}g}_u)}_u = \av{D^2_{z_0}g - \av{D^2_{z_0}g}_u}_u =$$ $$= \av{D^2_{z_0}g}_u - \av{\av{D^2_{z_0}g}_u}_u = \av{D^2_{z_0}g}_u - \av{D^2_{z_0}g}_u = 0.$$ The second equality follows since for a quadratic form $Q$ we have $D^2 Q = Q$.
\end{proof}

Note that by construction $D^2_{z_0}g = Hess_\HH(g)_{z_0} =
Hess_\HH(\widetilde{g})_{z_0} > \ep I$ where $\ep$ is provided by
Claim 2. We shall choose $r < r_0$ small enough so that $D^2_{(z_0 +
h)}g \geq \frac{\ep}{2} I$ for all $h \leq 2r$. Consider the Taylor
expansion of $D^2_{(z_0 + h)}g$ in $h$ about $h=0$: $$D^2_{(z_0 +
h)}g = D^2_{(z_0)}g + \Theta_1(h).$$ The function $\Theta_1$
satisfies $|\Theta_1| < \kappa |h| I$ for a constant $\kappa$
depending only on $\normC{\widetilde{g}}{10}{U}$. Therefore, by
Claim 1, $\kappa$ depends on $M,I,J,K,\Ome$ only. Consequently
$$D^2_{(z_0 + h)}g \geq \ep I - \kappa |h| I = (\ep - \kappa |h|)
I.$$ Hence for $|h| < \frac{\ep}{2\kappa}$ we have $D^2_{(z_0 + h)}g
\geq \frac{\ep}{2} I$. Therefore we can choose $r =
\min\{r_0,\frac{\ep}{2\kappa}\}/4$. Note that $r = const$. By the
choice of $r$, the function $g$ is convex with $D^2 g \geq
\frac{\ep}{2} I$ in $B(z_0,2r)$ with minimum in $z_0$, thus
satisfying condition 4.

By a straightforward computation we have then the estimate $g(z_0 +
h) > const \cdot \frac{\ep h^2}{2}$. Hence $a = const \cdot
\frac{\ep r^2}{2} = const$ can be chosen to satisfy condition 3.
Condition 2 is satisfied by Claim 3. The function $g$ can be
modified by adding a constant to satisfy condition 1. Condition 5 is
satisfied by property 3 of Claim 1.

\end{proof}

{\it Step 3: Wrapping up.}

Choose $z_0 \in M$ at which the function $\phi$ attains its minimum:
$\phi(z_0)= \min_M \phi$. Lemma \ref{good r, a and g} provides us
then with an appropriate $r,a >0$ and $g \in \sm{B(z_0,2r)}$. Then
the function $u = \phi + g$ in the domain $D = B(z_0,2r)$ satisfies
the conditions of Lemma \ref{key lemma} below. Indeed $\del\del_J u
= \Om + \del\del_J \phi > 0$ whence $u$ is plurisubharmonic and $u$
is negative on $D$ as both $\phi$ and $g$ are. The set $\{u < \inf_D
u + a\}$ is contained in $B(z_0,r) \Subset B(z_0,2r)$ and hence is
relatively compact.

Lemma \ref{key lemma} gives us $$\normL{\phi}{\infty}{M} -
\normL{g}{\infty}{D} \leq \normL{u}{\infty}{D} \leq a + const \cdot
(r/a)^{4n} \normL{u}{1}{D} \normL{f}{\infty}{D}^4.$$  Moreover
$\normL{u}{1}{D} \leq \normL{\phi}{1}{M} + \normL{g}{\infty}{D} \leq
const$ by Step 1 and property 5 of $g$ and $f > const \cdot f_0$ by
property 5 of $g$. Hence $$\normL{\phi}{\infty}{M} \leq const +
const \cdot \normL{f}{\infty}{D}^4,$$ which finishes the proof.
\end{proof}

%%%%%%%%%%%%%%%%%%%%%%%%%%%%%%%%%%%%%%%%%%%%%%%%%%%%%%%%%%%%%%%%%%%%%%%%%%%%%%%%%%%%%%%%%%%%%%%
%%%%%%%%%%%%%%%%%%%%%%%%%%%%%%%%%%%%%%%%%%%%%%%%%%%%%%%%%%%%%%%%%%%%%%%%%%%%%%%%%%%%%%%%%%%%%%%

\subsection{The ABP inequality}

We would like to use the following lemma in the proof.

\begin{lemma}\label{key lemma}
Let $D \subset \HH^n$ be a bounded domain. Let $u$ be a negative
$C^2$ strictly plurisubharmonic function in $D$ and $a > 0$ a
constant such that the set $\{u < \inf_D u + a\}$ is relatively
compact in $D$. Denote by $f$ the Moore determinant $f = det (
\frac{\del^2 u}{\del \bar q_i \del q_j})$. Then
$$\normL{u}{\infty}{D} \leq a + const \cdot
\frac{\mathrm{diam}(D)^{4n}}{a^{4n}} \normL{u}{1}{D}
\normL{f}{\infty}{D}^4.$$
\end{lemma}

This lemma is a formal consequence of the following proposition:

\begin{prop}\label{key proposition}
Let $D \subset \HH^n$ be a bounded domain. Let $u\in C^2(D)\cap
C(\overline{D})$ be a non-positive strictly plurisubharmonic
function in $D$, that vanishes on $\del D$. Denote by $f$ the Moore
determinant $f = det ( \frac{\del^2 u}{\del\bar q_i \del
q_j})$. Then
$$\normL{u}{\infty}{D} \leq const \cdot \mathrm{diam}(D)
\normL{f}{4}{D}^{1/n}.$$
\end{prop}

We will first show how the lemma follows from the proposition and then prove the latter.

\begin{proof}(Lemma \ref{key lemma})
Define $v:= u -\inf_D u -a$. Set $D'=\{v < 0\}$. Then $D'\Subset D$
by assumption. By Proposition \ref{key proposition},
$$a = \normL{v}{\infty}{D'} \leq const \cdot \mathrm{diam}(D') \normL{f}{4}{D'}^{1/n} \leq const \cdot \mathrm{diam}(D')(Vol(D'))^{1/4n} \normL{f}{\infty}{D'}^{1/n}.$$
On the other hand $$Vol(D') \leq \frac{\normL{u}{1}{D}}{|\inf_D u + a|} = \frac{\normL{u}{1}{D}}{\normL{u}{\infty}{D} - a}.$$ The lemma now follows by
direct substitution.
\end{proof}

\begin{proof}(Proposition \ref{key proposition}) By the Alexandrov-Bakelman-Pucci inequality (Lemma 9.2 in \cite{GilbargTrudinger})
we have
\begin{equation}\label{Alexandrov-Bakelman-Pucci}\normL{u}{\infty}{D} \leq const \cdot \mathrm{diam}(D) (\int_\Gamma det D^2 u)^{1/4n},\end{equation}
where $$\Gamma := \{x \in D |\; u(x) + \langle Du(x),y-x \rangle
\leq u(y) \;\forall y\in D\} \subset \{D^2 u \geq 0\}.$$

Using quaternionic linear algebra (cf. \cite{alesker-bsm} and the
references therein) if $w^1,...,w^n$ are a hyper-Hermitian orthonormal $\HH$-base of eigenvectors of
$\frac{\del^2 u}{\del\bar q_i \del q_j}$ then
$w^1,...,w^n,w^1\circ I,...,w^n\circ I,w^1\circ J,...,w^n\circ
J,w^1\circ K,...,w^n \circ K$ form an orthonormal basis in $\R^{4n}
\cong \HH^n$ and at a point where $D^2u \geq 0$ we obtain
$$det(\frac{\del^2 u}{\del\bar q_i \del q_j}) = \prod_{l=1}^n
\sum_{i,j =1}^n \bar w^l_i \frac{\del^2 u}{\del\bar q_i \del q_j}
w^l_j = const \cdot \prod_{l=1}^n (D^2 u (w^l) + D^2 u
(w^l \circ I) + D^2 u (w^l \circ J) + D^2 u (w^l \circ K))$$ $$\geq
const \prod_{l=1}^n \sqrt[4]{(D^2 u (w^l)) (D^2 u (w^l \circ I)) (D^2 u (w^l \circ
J)) (D^2 u (w^l \circ K))} \geq const \sqrt[4]{det D^2 u}.$$ The
last inequality follows from the fact that the determinant of a real
nonnegative symmetric matrix does not exceed the product of its
diagonal entries. Thus at a point where $D^2 u \geq 0$ we have
\begin{equation}\label{det D^2 and det quaternionic Hess} det D^2 u
\leq const \cdot det(\frac{\del^2 u}{\del\bar q_i \del
q_j})^4.\end{equation} Plugging this into
(\ref{Alexandrov-Bakelman-Pucci}) yields the proposition.
\end{proof}

\subsection{The $L^1$-bound}\label{L1 bound section}

Let $(M^{4n},I,J,K)$ be a compact hypercomplex manifold. Let us fix
$\Ome\in \Lam^{2,0}_{I,\RR}(M)$ which is real and strictly positive (we do
{\itshape not} assume that $\Ome$ is HKT, i.e. $\pt\Ome$ may not
vanish). The following result is a version of the Gauduchon theorem
\cite{gauduchon-77}; the proof follows closely the lines of his arguments.

\begin{prop}\label{P:gauduchon}
There exists a unique, up to a positive multiplicative constant,
form $\Theta \in \Lam^{2n,0}_{I,\RR}$ which is strictly positive pointwise (in
particular non-vanishing) and
$$\pt\pt_J(\Ome^{n-1}\wedge\bar\Theta)=0.$$
\end{prop}
\begin{proof} Let us fix a positive non-vanishing form $\Theta_0\in
\Lam^{2n,0}_{I,\RR}$ (e.g. $\Ome^n$). Let us consider the operator $A$ on real functions
on $M$ given by
\begin{eqnarray}\label{D:A}
A f:=\frac{\pt\pt_J f\wedge
\Ome^{n-1}\wedge\bar\Theta_0}{\Ome^n\wedge\bar\Theta_0}.
\end{eqnarray} Clearly $A$ is elliptic and $A(1)=0$.

Define a positive definite scalar product on functions by
\begin{eqnarray}
<f,g>=\int_M fg\cdot \Ome^{n}\wedge\bar \Theta_0.
\end{eqnarray}
Let us compute the conjugate operator $A^*$ with respect to this
product:
\begin{eqnarray*}
<Af,g>=\int (\pt\pt_Jf) \cdot g\wedge \Ome^{n-1}\wedge\bar
\Theta_0=\int f\cdot \pt\pt_J(g\cdot \Ome^{n-1}\wedge\bar
\Theta_0)=<f, \frac{\pt\pt_J(g\cdot \Ome^{n-1}\wedge\bar
\Theta_0)}{\Ome^{n}\wedge\bar \Theta_0}>.
\end{eqnarray*}
Hence
\begin{eqnarray}\label{A*}
A^*g=\frac{\pt\pt_J(g\cdot \Ome^{n-1}\wedge\bar
\Theta_0)}{\Ome^{n}\wedge\bar \Theta_0}.
\end{eqnarray}

Since the operators $A$ and $A^*$ have the same symbol, their
indices vanish. By the maximum principle the kernel of $A$ consists
only of constant functions. Hence the kernel of $A^*$ is one
dimensional. Let us denote by $G$ a generator of the kernel of
$A^*$. If we show that $G$ never vanishes then $\Theta:=\pm G\cdot
\Theta_0$ will satisfy the proposition.

\begin{lemma}\label{L:integral}
$$<G,1>\ne 0.$$
\end{lemma}
\begin{proof} Assume the contrary. Then $1\in (Ker A^*)^\perp=Im A$.
Then there exists a function $f$ such that $Af=1$. Let $x_0$ be a
point of maximum of $f$. Then since $A(1)=0$ we get
$$1=Af(x_0)\leq 0.$$
This is a contradiction. The lemma is proved. \end{proof}

\hfill

Let us now show that $G$ cannot change sign. Let $\phi>0$ be any
positive function. Put $\Ome_\phi:=\phi\cdot \Ome$.
Define the operator $A_\phi$ exactly as $A$ but with $\Ome_\phi$
instead of $\Ome$:
\begin{eqnarray*}
A_\phi f=\frac{\pt\pt_J f\wedge
\Ome^{n-1}_\phi\wedge\bar\Theta_0}{\Ome^n_\phi\wedge\bar\Theta_0}=\phi^{-1}\cdot
Af.
\end{eqnarray*}
Let us compute the adjoint $A_\phi^*$ with respect to the pairing
$<\cdot,\cdot>_\phi$ which is defined by
$$<f,g>_\phi=\int_M fg \cdot \Ome^{n}_\phi\wedge\bar \Theta_0.$$
Similarly to (\ref{A*}), we have
$$A_\phi^* g=\frac{\pt\pt_J(g\cdot \Ome^{n-1}_\phi\wedge\bar
\Theta_0)}{\Ome^{n}_\phi\wedge\bar \Theta_0}.$$

By assumption $A^*G=0$, namely $\pt\pt_J(G\cdot \Ome^{n-1}\wedge\bar
\Theta_0)=0$. Trivially
$$0=\pt\pt_J(\phi^{1-n}G\cdot \Ome^{n-1}_\phi\wedge\bar
\Theta_0)=A^*_\phi(\phi^{1-n}\cdot G)\cdot
(\Ome^n_\phi\wedge\bar\Theta_0).$$ Hence $\phi^{1-n}G\in Ker
A_\phi^*$.

Let us apply Lemma \ref{L:integral} to $\phi^{1-n}\cdot G$ instead
of $G$ and to $<\cdot,\cdot>_\phi$ instead of $<\cdot,\cdot>$. We
have
$$0\ne <\phi^{1-n}\cdot G,1>_\phi=\int \phi^{1-n} G\cdot \Ome^n_\phi\wedge\bar \Theta_0=\int \phi\cdot G\cdot
\Ome^n\wedge\bar\Theta_0.$$ Thus we have shown that the function $G$
is such that for any $\phi>0$
$$\int \phi G\cdot (\Ome^n\wedge\bar\Theta_0)\ne 0.$$
This implies easily that $G$ cannot change its sign.

\hfill

It remains to show that $G$ cannot vanish at any point. This
immediately follows from the following general lemma.
\begin{lemma}[\cite{gauduchon-77}, Lemma 2]\label{L:general-pde}
Any non-negative solution of a linear elliptic differential equation
of second order with $C^\infty$-smooth real coefficients is either
strictly positive or vanishes identically.
\end{lemma}
Thus Proposition \ref{P:gauduchon} is proved. \end{proof}

\hfill

\begin{lemma}\label{L:green}
Let $(M^{4n}, I,J,K)$ be a compact hypercomplex manifold with an
HKT-form $\Ome$. Let $\Theta\in \Lam_{I,\RR}^{2n,0}$ be a positive
form as in Proposition \ref{P:gauduchon}, i.e.
$\pt\pt_J(\Ome^{n-1}\wedge \bar\Theta)=0$. Let
$$A\phi=\frac{\pt\pt_J\phi\wedge \Ome^{n-1}\wedge \bar\Theta}{\Ome^n\wedge\bar\Theta}.$$
Then the operator $A$ admits a non-negative Green function
$G(x,y)\geq 0$, namely
$$-\int G(x,y)\cdot A\phi(y)\cdot(\Ome^n\wedge\bar\Theta)=\phi(x)-\frac{1}{\text{vol}(M)} \int \phi\cdot(\Ome^n\wedge\bar\Theta)$$
for any function $\phi$ and point $x$ (here $\text{vol}(M)=\int \Ome^n\wedge\bar\Theta$).
\end{lemma}
\begin{proof} Let $G_0$ be a Green function bounded from below (as in Appendix \ref{Green function section}). However we can add to $G_0$ a large
constant. Indeed
\begin{eqnarray*}
\int A\phi(y)\cdot(\Ome^n\wedge\bar\Theta)=\int \pt\pt_J\phi \wedge
\Ome^{n-1}\wedge\bar \Theta=\int \phi \cdot
\pt\pt_J(\Ome^{n-1}\wedge\bar \Theta)=0.
\end{eqnarray*}
\end{proof}

\begin{prop}\label{P:L1}
Assume that $\Ome +\pt\pt_J\phi >0$ and $\max_M\phi=0$. Then
$||\phi||_{L^1}$ is bounded by a constant depending on
$(M,I,J,K,\Ome)$ only.
\end{prop}
\begin{proof} Let us choose positive $\Theta$ as in Proposition
\ref{P:gauduchon}. Let $x\in M$ be a point of maximum of $\phi$. Denote
by $A$ the operator as in Lemma \ref{L:green}, and let $G\geq0$ be
its Green function. We have by assumption
$$-A\phi=\frac{-\pt\pt_J\phi\wedge \Ome^{n-1}\wedge
\bar\Theta}{\Ome^n\wedge\bar\Theta}\leq 1.$$

We have
\begin{eqnarray*}
||\phi||_{L^1}=0-\int \phi \cdot (\Ome^n\wedge\bar\Theta)=-\int
G(x,y)A\phi(y)\cdot (\Ome^n\wedge\bar\Theta)\leq\\ \int G(x,y)\cdot
(\Ome^n\wedge\bar\Theta)\leq const(M,\Ome).
\end{eqnarray*}
\end{proof}

%%%%%%%%%%%%%%%%%%%%%%%%%%%%%%%%%%%%%%%%%%%%%%%%%%%%%%%%%%%%%%%%%%%%%%%%%%%%%%%%%%%%%%%%%%%%%%%
%%%%%%%%%%%%%%%%%%%%%%%%%%%%%%%%%%%%%%%%%%%%%%%%%%%%%%%%%%%%%%%%%%%%%%%%%%%%%%%%%%%%%%%%%%%%%%%
\appendix
\section{Green function bounded from below}\label{Green function section}

We adapt several proofs from \cite{ShubinSpectralElliptic}
together with a classical local result to prove that for any
elliptic second-order differential operator $A$ on a compact
manifold $M$ there exists a Green function $G(x,y)$ that satisfies
the properties below. This result seemed plausible to experts, however
we have not found a proof in the literature and hence choose to fill this gap here.

Consider the Hilbert space $L^2(M,\R)$ where the inner product on
smooth functions is given by $\langle f,g \rangle : = \int_M
f(x)g(x) dV(x)$ for a normalized smooth measure $dV$. Denote by
$U:L^2(M,\R) \to L^2(M,\R)$ the operator of projection onto the
orthogonal complement $J_0$ of the kernel $K \subset \sm{M}$ of $A$. Then the
properties are:

\begin{enumerate}
\setcounter{enumi}{-1}
\item $\int_M G(x,y)A\phi(y)dV(y) = U \phi(x)$, for any $\phi \in
\sm{M}$.
\item $G(x,y)$ is smooth outside the diagonal $\Delta \subset M\times M$.
\item $G(x,y) \geq -D_1$ (a.e.) on $M\times M$ for a constant $D_1 >0$.
\item For any fixed $x \in M$, $\normL{G(x,\cdot)}{1}{M} < D_2$ for a constant $D_2 >0$.
\end{enumerate}

Note that these properties are what is left to show to establish our
desired a priori $L^1$-bound. Indeed in our case $U\phi=\phi -
\int_M \phi dV$, for $dV$ the normalized smooth measure  $dV = \Om^n
\wedge \Theta_0^n / \int_M \Om^n \wedge \Theta_0^n$.

%\begin{remark}
%Note that $P\phi:=\phi - \int_M \phi dV$ is the projection operator
%$P:L^2(M,\R) \to L^2(M,\R)$ onto the orthogonal complement
%$V^{\perp}$ of the kernel $V$ of $A$, where the inner product is
%given by $\langle f,g \rangle : = \int_M f(x)g(x) dV(x)$. Note that
%as $A|_{V^{\perp}} : V^{\perp} \to L^2(M,\R)$ is injective, there
%exists a continuous operator $A^\iota: L^2(M,\R) \to L^2(M,\R)$ such
%that $A^\iota A = P$. Moreover, by the elliptic regularity estimates
%(Lemma 1.4 in \cite{ShubinSpectralElliptic}) that are called the
%Calderon-Zygmund theory elsewhere in the literature and the closed
%graph theorem, $A^\iota$ is continuous as an operator $W^s_2(M,\R)
%\to W^{s+2}_2(M,\R)$. Moreover, the Green function we are interested
%in is nothing more than the Schwartz kernel of $A^\iota$.
%\end{remark}

We now construct an operator $A^\iota: \sm{M,\R} \to \sm{M,\R}$
whose Schwartz kernel (c.f. \cite[Section 5.2]{HormanderI}) satisfies
property $(0)$. We then show that it also satisfies the other
properties.

First consider the spaces $K=Ker A \subset \sm{M,\R}$ and $C = Ker
A^* \subset \sm{M,\R}$. These have additive closed complements $I = Image
A$ and $J = Image A^*$ where $A^*$ is the adjoint differential
operator to $A$ with respect to $dV$. For each Sobolev completion $L^2_s$ of $\sm{M,\R}$ denote by $I_s$,
$J_s$ the corresponding completions of subspaces. These are the images of the
corresponding operators $A_{s+2}:L^2_{s+2} \to L^2_s$ and $(A^*)_{s+2}:L^2_{s+2} \to L^2_s$.
For what follows choose an $L^2$-orthonormal basis $\{f_1,...,f_k\} \subset \sm{M}$ of the kernel $K$ of $A$,
and an $L^2$-orthonormal basis $\{g_1,...,g_l\} \subset \sm{M}$ of the kernel $C$ of $A^*.$ Denote by $P$ the operator $\sm{M} \to \sm{M}$ given by $f \mapsto \Sum_j \langle f,f_j \rangle f_j$. Its image is $K$ and it is a smoothing operator - its Schwartz kernel is $K_P(x,y) = \Sum_j f_j(x)f_j(y) \in \sm{M \times M}$. Denote by $P_s$ its extension $L^2_s \to \sm{M}$. Similarly define $Q$ with image $C$, with Schwartz kernel $K_Q(x,y) = \Sum_l g_l(x)g_l(y) \in \sm{M \times M}$ and $Q_s$ its extension to $L^2_s$.

For every $s$, we have $L^2_s = K \bigoplus J_s$ and $L^2_s = C \bigoplus I_s$ as Banach spaces. Denote by $U_s$ and $V_s$ the projections complementary to $P_s$ and to $Q_s$ - namely $U_s = Id - P_s$ and $V_s = Id - Q_s$. Noting that $A_s|_{J_s}:J_s \to I_{s-2}$ is a bijective continuous map of Hilbert spaces\footnote{Indeed, a differential operator of order $2$ induces a continuous linear map of Hilbert spaces $L^2_s \to L^2_{s-2}$, and $J_s$ being a closed direct complement to the kernel of this map, the conclusion follows.} and has as such, by the Banach inverse mapping theorem, a continuous inverse $D_{s-2}:I_{s-2} \to J_s$, we construct a continuous linear map \[(A^\iota)_{s-2}:L^2_{s-2} \to L^2_{s}\] as the composition
\[L^2_{s-2} \xrightarrow{V_{s-2}} I_{s-2} \xrightarrow{D_{s-2}} J_s \hookrightarrow L^2_s.\] It is easy to check that these diagrams for different $s$ are commuting with the embeddings $L^2_{s'} \to L^2_s$ for $s' > s$ and therefore define a continuous linear operator  \[A^\iota:\sm{M} \to \sm{M}.\]  We note that as $(A^\iota)_{s-2} (A)_s = U_s$ for all $s$, we have $A^\iota A = U$ on the $C^\infty$ level - i.e. property (0).

%\[\xymatrix{
%FX \ar[r]^-{Ff} \ar[d]_{\eta_X} & FY \ar[d]^{\eta_Y} \\
%GX \ar[r]_-{Gf} & GY\\} \]

We now construct a pseudodifferential parametrix $B$ for $A$ and then show that $B$ and $A^\iota$ differ by a smoothing operator.

\begin{construction}
It is well-known (c.f. \cite[Proposition 3.4]{ShubinSpectralElliptic}) that for the elliptic differential operator $A$ there exists a pseudo-differential operator $B$ on $M$ such that $AB-Id$ and $BA - Id$ are smoothing pseudo-differential operators. Such an operator $B$ is called a (pseudo-differential) parametrix for $A$. For the convenience of the reader and for use in establishing properties $(2)$ and $(3)$ we recall the construction of this operator. Cover the manifold $M$ by small balls $B(x_i,\eps)$ with respect to an auxiliary Riemannian metric. Consider the restrictions $A_i$ of $A$ to the space of functions compactly supported in $B(x_i,2\eps)$. Then each operator $A_i$ has a pseudo-differential parametrix $B_i$. Then \begin{equation}\label{B construction}B = \Sum_i \Psi_i B_i \Phi_i\end{equation} is the required parametrix \cite[Proposition 3.4]{ShubinSpectralElliptic}. Here $\Phi_i,\Psi_i$ are the operators of multiplication by $\phi_i,\psi_i$ respectively for $\{\phi_i\}_i$ forming a partition of unity subordinate to the covering $\{B(x_i,\eps)\}_i$ and $\psi_i$ is a positive smooth function supported in $B(x_i,1.5 \eps)$ and identically equal to $1$ on $B(x_i,\eps)$.
\end{construction}

\begin{lemma}\label{local pseudodifferential}
The uniformly elliptic operator $L=A_i$ on the ball $\Om=B(x_i,2\eps)$ has a pseudodifferential
parametrix $B_i$ whose Schwartz kernel $G_i(x,y)$ satisfies
\[G(x,x+z) =G_0(x,x+z) + o(|z|^{2-n})\] uniformly for $|z|<1$, $x \in K \Subset \Om$ as $z \to 0$, where
$$G_0(x,y) = \frac{1}{(n-2) \om_n} det(L_2(x))^{-1/2} \langle L_2(x)(x -
y),(x - y)\rangle^{2-n},$$ $$Lu = -\Sum_{i,j}
(L_2)_{ij}\del_i \del_j u + \Sum b_k \del_k u + c u,$$ for a uniformly positive definite\footnote{The sign in the definition of $L$ appears already in the case of the constant coefficient Laplace operator.} matrix $(L_2)(x)$ and $\om_n =
\frac{2 \pi^{n/2}}{\Gamma(n/2)}$.
\end{lemma}

This can be shown in several ways - one is by \cite[Chapter III, Theorem 3.3]{NorioShimakura} together with an explicit calculation (up to a smooth function) of the leading term as a Fourier transform of a locally integrable homogenous function.

By Lemma \ref{local pseudodifferential} the functions $G_i$ are all bounded from below on $B(x_i,1.5 \eps)\times B(x_i,1.5 \eps)$ since they are positive near the diagonal and smooth off the diagonal. Now the Schwarz kernel of $B$ is $$K'(x,y)=\Sum_i \psi_i(x) G_i(x,y)
\phi_i(y).$$ We would like to provide lower bounds on $K'(x,y)$ and upper bounds on $\normL{K'(x,\cdot)}{1}{M}$.

\begin{lemma}Since all $G_i(x,y)$ are bounded from below, we have $K'(x,y) \geq -D'_1$ for $D'_1>0$ on $M\times M$. Moreover since for all $G_i(x,y)$ we have $\normL{G_i(x,\cdot)}{1}{B(x_i,\eps)} < D'_{2,i}$ we have  $\normL{K'(x,\cdot)}{1}{M} < D'_2 = \Sum_i D'_{2,i}$.\end{lemma}

We now compare $A^\iota$ and $B$. We follow \cite[Theorem 3.6]{ShubinSpectralElliptic} to show the following.

\begin{prop} \begin{equation}\label{A iota via B and T}A^\iota = B + T\end{equation} where $T$ has a Schwartz kernel $K_T \in \sm{M\times M}$. \end{prop}

Indeed we have $AB=Id + R$, for a smoothing operator $R$. Multiplying on the left by $A^\iota$ we have $UB = A^\iota + A^\iota R$, that is $A^\iota = B - PB - A^\iota R$, where $P$ is the projection operator onto $K = Ker A$.

Note first that $A^\iota R$ is smoothing as $R$ is smoothing and
$A^\iota$ is continuous as an operator $\sm{M,\R} \to \sm{M,\R}$.
Indeed smoothing operators can be described on one hand as those
operators $\sm{M,\R} \to \mathcal{D}(M,\R)$ that have smooth
kernels, and on the other as those operators that can be extended to
linear maps $\mathcal{D}(M,\R) \to \sm{M,\R}$ continuous in the weak topology on every
bounded subset of $\mathcal{D}(M,\R)$. We refer to \cite[(23.11.1)]{Dieudonne} for the equivalence of
these descriptions. From the second description, the statement is
immediate.

Note also that $PB$ is smoothing since $B$ is a pseudo-differential operator and $P$ is smoothing. Hence $A^\iota = B + T$ for $T = -PB - A^\iota R$ - a smoothing operator.

\begin{wrappingup}Now note that as $A^\iota = B + T$, we have $G(x,y) = K'(x,y) + K_T(x,y)$, yet $K'(x,y) \geq -D'_1$ on $M\times M$ and $K_T(x,y) \geq -D''_1$ as $K_T \in \sm{M \times M}$. Hence $G(x,y) \geq -D_1$ for $D_1 = D'_1 + D''_1$ and property (1) is established. Property (0) follows by definition and property (2) follows from $\normL{K'(x,\cdot)}{1}{M} < D'_2$ and $\normL{K_T(x,\cdot)}{1}{M} \leq E = \normL{K_T}{\infty}{M\times M} < \infty$ for $D_2 = E + D'_2$.
\end{wrappingup}

\bibliographystyle{amsplain}
\bibliography{UniformBoundsCalabiYauHKTRefs1}

\end{document}